\documentclass[preprint,12pt]{elsarticle}

\usepackage{amscd}
\usepackage{amsmath}
\usepackage{amsthm}
\usepackage{amsfonts}
\usepackage{graphicx}
\usepackage{amssymb}
\usepackage{color}
\usepackage{amsbsy}
\usepackage{graphicx}
\usepackage{amssymb,color,amsbsy}

\usepackage{float}
\usepackage{comment}

\usepackage[ruled]{algorithm2e}

\usepackage[matrix,arrow]{xy}
\usepackage{mathabx}   

\DeclareMathAlphabet{\mathpzc}{OT1}{pzc}{m}{it}

\usepackage{pgfpages}
\usepackage{tikz}
\usetikzlibrary{shapes.geometric}
\usetikzlibrary{decorations.pathmorphing}
\usetikzlibrary{arrows}
\usetikzlibrary{backgrounds}
\usetikzlibrary{positioning}
\usetikzlibrary{fit}
\usepackage{caption}
\usepackage{amsthm}

\usepackage{amsmath}

\usepackage[pagebackref=true, bookmarksopen=true,colorlinks=true, linkcolor=red,citecolor=blue]{hyperref}

\usepackage[capitalise]{cleveref}  

\global\long\def\s{\subset}%

\global\long\def\P{\prime}%

\global\long\def\ñ{\sim}%

\usepackage{tikz-cd}

\usepackage{tkz-graph}

\usepackage{multicol}

\usepackage{subcaption}

\newtheorem{theorem}{Theorem}[section]

\newtheorem{corollary}[theorem]{Corollary}

\newtheorem{lemma}[theorem]{Lemma}

\newtheorem{problem}[theorem]{Problem}

\newtheorem{observation}[theorem]{Observation}

\newcommand{\Sachs}[1]{\operatorname{Sachs}(#1)}

\usepackage{comment}
\begin{document}

\begin{abstract}
	A graph is said to be a Sterboul--Deming graph if $KE(G)=\emptyset$, that is, if every vertex of $G$ belongs to a posy or a flower (structures introduced by Sterboul, Deming, and Edmonds). These graphs can be regarded as the structural counterparts of K\"onig--Egerv\'ary graphs. In this paper, we present several characterizations of Sterboul--Deming graphs. We first study the case of graphs with a perfect matching and with a unique perfect matching, providing a constructive algorithm to obtain the decomposition $(SD(G), KE(G))$. Then, we extend the analysis to the general case through the Gallai--Edmonds decomposition. In addition, we show that the class of Sterboul--Deming graphs is remarkably broad: it contains all graphs having a $\{C_n : n \textnormal{ odd}\}$-factor, providing a simple structural criterion for identifying such graphs. These results establish new connections between classical decomposition theorems and the internal structure of non--K\"onig--Egerv\'ary graphs.
\end{abstract}


\begin{frontmatter}
	
	\title{Sterboul–Deming Graphs: Characterizations}


	\author[pan,daj]{Kevin Pereyra}
	\ead{kdpereyra@unsl.edu.ar}

	\address[pan]{Universidad Nacional de San Luis, Argentina.}
	\address[daj]{IMASL-CONICET, Argentina.}
	
	\date{Received: date / Accepted: date}

\end{frontmatter}
%
%
%

\section{Introduction}

Let $\alpha(G)$ denote the cardinality of a maximum independent set,
and let $\mu(G)$ be the size of a maximum matching in $G=(V,E)$.
It is known that $\alpha(G)+\mu(G)$ equals the order of $G$,
in which case $G$ is a König--Egerváry graph 
\cite{deming1979independence,gavril1977testing,stersoul1979characterization}.
Various properties of König--Egerváry graphs were presented in 
\cite{bourjolly2009node,jarden2017two,levit2006alpha,levit2012critical}.
It is known that every bipartite graph is a König--Egerváry graph 
\cite{egervary1931combinatorial}.

The term subgraph in here is understood as a subgraph defined by a
graph and a given matching in that graph. In \cite{edmonds1965paths}, Edmonds introduced the following concepts
relative to a matching $M$ of a graph $G$ and its subgraphs. An
$M$-blossom of $G$ is an odd cycle of length $2k+1$ with $k$ edges
in $M$. The vertex not saturated by $M$ in the cycle is called
the \textit{base} of the blossom. An $M$-stem is an $M$-alternating
path of even length (possibly zero) connecting the base of the blossom
with a vertex not saturated by $M$ in $G$. The base is the only common vertex
between the blossom and the stem. An $M$-Tflower is a blossom joined
with a stem. The vertex not saturated by $M$ in the stem is called
the \textit{root} of the flower. In \cite{jaume2025confpart2}, this concept is generalized by introducing the notion of an \(M\)-Jflower, obtained from an \(M\)-Tflower by allowing its stem to be an \(mm\)-\(M\)walk instead of an \(mm\)-\(M\)path.

In \cite{stersoul1979characterization}, Sterboul introduced the concept
of a \textit{posy} for the first time.

\begin{itemize}
	\item Un $M$-Tposy consists of two $M$-blossoms joined by an $mm$-$M$path.
	The endpoints of the path are the bases of the two blossoms. There
	are no internal vertices of the path in the blossoms. The intersection
	between the blossoms can only be an $mm$-$M$path (in other words,
	a subdivision of a barbell or of a $K_{4}$ with a perfect matching
	inherited from $G$).
	\item Un $M$-Sposy consists of two (not necessarily disjoint) $M$-blossoms
	joined by an $mm$-$M$path. The endpoints of the path are the bases
	of the two blossoms. There are no internal vertices of the path in
	the blossoms. 
	\item Un $M$-Jposy consists of two (not necessarily disjoint) $M$-blossoms
	joined by an $mm$-$M$walk. The endpoints of the walk are the bases
	of the two blossoms. 
\end{itemize}

Of course, an $M$-Tposy is an $M$-Sposy, and an $M$-Sposy is an $M$-Jposy.

\begin{theorem}\label{safe}
	For a graph $G$, the following properties are equivalent: 
	\begin{itemize}
		\item $G$ is a non-König--Egerváry graph.
		\item For every maximum matching $M$, there exists an $M$-Tflower or an $M$-Sposy
		in $G$.
		\item For some maximum matching $M$, there exists an $M$-Tflower or an $M$-Tposy
		in $G$.
	\end{itemize}
\end{theorem}

Sterboul \cite{stersoul1979characterization} was the first to characterize
König--Egerváry graphs via forbidden configurations relative to
a maximum matching. Subsequently, Korach, Nguyen,
and Peis \cite{korach2006subgraph} reformulated this characterization
in terms of simpler configurations, unifying the structures of
flowers and posies. Later, Bonomo et al. \cite{bonomo2013forbidden}
obtained a purely structural characterization based on forbidden subgraphs.
More recently, in \cite{jaume2025confpart3,jaume2025confpart2,jaume2025confpart1},
results were obtained that simplify working with flower and posy
structures.

The set of vertices of $G$ lying in a Tflower or Tposy, for any maximum matching,
is denoted by $SD(G)$, and we write $KE(G)=V(G)-SD(G)$.
The sets $SD(G)$ and $KE(G)$ constitute the SD--KE decomposition of the graph. \cref{12ews} provides a very useful characterization of this decomposition.
\begin{equation*}
	\begin{aligned}
		SD(G) = \{\,v \in V(G) :\;& \text{there exists a matching } M \text{ such that } \\
		& v \text{ lies in an } M\text{-Tposy or in an } M\text{-Tflower}\,\},
	\end{aligned}
\end{equation*}
\begin{equation*}
	\begin{aligned}
		SDJ(G) = \{\,v \in V(G) :\;& \text{there exists a matching } M \text{ such that } \\
		& v \text{ lies in an } M\text{-Jposy or in an } M\text{-Jflower}\,\}.
	\end{aligned}
\end{equation*}

\begin{theorem}[\cite{jaume2025confpart1}\label{12ews}]
	For every graph $G$,
	$$SD(G)=JSD(G).$$
\end{theorem}

\noindent That is, the set $SD(G)$ does not depend on the particular definition of posy or flower we choose.

A graph $G$ is called a Sterboul–Deming graph (SD graph) if $KE(G)=\emptyset$. Essentially, it can be regarded as the structural counterpart of a König–Egerváry graph. From \cref{12ews}, we immediately obtain the following characterization:

\begin{theorem}\label{asd1s}
	A graph is a Sterboul–Deming graph if and only if every vertex belongs to an $M$-Jposy or an $M$-Jflower, for some maximum matching $M$.
\end{theorem}

\cref{asd1s} will be a central tool throughout this work, as it allows us to considerably simplify the study of Sterboul–Deming graphs. When writing proofs, it is much more tedious to work with Tposies or Tflowers because of the intersection constraints they impose; in contrast, the use of Jposies and Jflowers provides much greater flexibility.

The main goal of this paper is to characterize Sterboul–Deming graphs. In \cref{preliminaries} we present a brief review of the notation used throughout the paper. In \cref{sss1} we provide several characterizations for Sterboul–Deming graphs with a perfect matching and with a unique perfect matching, together with an algorithm to obtain the SD–KE decomposition in graphs with a unique perfect matching. In \cref{sss2} we extend the study to the general case through the Gallai–Edmonds decomposition. 
Beyond these characterizations, we show that the class of Sterboul–Deming graphs is surprisingly broad: it includes all graphs having a $\{ C_n : n\textnormal{ odd}\} $-factor and, in particular, all odd 2-factored graphs, thus providing a simple and effective criterion to verify whether a graph belongs to this class. We conclude the work by leaving some open problems.

\section{Preliminaries}\label{preliminaries}

All graphs considered in this paper are finite, undirected, and simple.
For any undefined terminology or notation, we refer the reader to
Lovász and Plummer \cite{LP} or Diestel \cite{Distel}.

Let \( G = (V, E) \) be a simple graph, where \( V = V(G) \) is the finite set of vertices and \( E = E(G) \) is the set of edges, with \( E \subseteq \{\{u, v\} : u, v \in V, u \neq v\} \). We denote the edge \( e=\{u, v\} \) as \( uv \). A subgraph of \( G \) is a graph \( H \) such that \( V(H) \subseteq V(G) \) and \( E(H) \subseteq E(G) \). A subgraph \( H \) of \( G \) is called a \textit{spanning} subgraph if \( V(H) = V(G) \). 

Let \( e \in E(G) \) and \( v \in V(G) \). We define \( G - e := (V, E \setminus \{e\}) \) and \( G - v := (V \setminus \{v\}, \{uw \in E : u,w \neq v\}) \). If \( X \subseteq V(G) \), the \textit{induced} subgraph of \( G \) by \( X \) is the subgraph \( G[X]=(X,F) \), where \( F:=\{uv \in E(G) : u, v \in X\} \).

Given a vertex set $S \subseteq V(G)$, we denote by $\partial(S)$
the set of edges with one endpoint in $S$ and the other in $V(G)\setminus S$.

A \textit{matching} \(M\) in a graph \(G\) is a set of pairwise non-adjacent edges. The \textit{matching number} of \(G\), denoted by  \(\mu(G)\), is the maximum cardinality of any matching in \(G\). Matchings induce an involution on the vertex set of the graph: \(M:V(G)\rightarrow V(G)\), where \(M(v)=u\) if \(uv \in M\), and \(M(v)=v\) otherwise. If \(S, U \subseteq V(G)\) with \(S \cap U = \emptyset\), we say that \(M\) is a matching from \(S\) to \(U\) if \(M(S) \subseteq U\). A matching $M$ is \emph{perfect} if $M(v)\neq v$ for every vertex
of the graph. A matching is \emph{near-perfect} if \( \left|{v \in V(G) : M(v) = v}\right| = 1 \).  A graph is a factor-critical graph if $G-v$ has perfect matching
for every vertex $v\in V(G)$. 

A vertex set \( S \subseteq V \) is \textit{independent} if, for every pair of vertices \( u, v \in S \), we have \( uv \notin E \). The number of vertices in a maximum independent set is denoted by \( \alpha(G) \). By definition, the empty set is independent. Let $\Omega^{*}(G)=\{S\s V(G):\textnormal{ S is an independendent set}\}$
and $\Omega(G)=\{S\in\Omega^{*}(G):\left|S\right|=\alpha(G)\}$.

Let $M$ be a matching in $G$. A path (or walk) is called \textit{alternating} with respect to $M$ if, for each pair of consecutive edges in the path, exactly one of them belongs to $M$. If the matching is clear from context, we simply say that the path is alternating. Given an alternating path (or walk), we say that $P$ is: \textit{\(mm\)-\(M\)path} if it starts and ends with edges in $M$, \textit{\(nn\)-\(M\)path} if it starts and ends with edges not in $M$, \textit{\(mn\)-\(M\)path} if it starts with an edge in $M$ and ends with one not in $M$, and \textit{\(nm\)-\(M\)path} if it starts with an unmatched edge and ends with a matched edge. 

$i(G)$ denotes the number of isolated vertex in the graph $G$. 

\section{Sterboul-Deming graphs with perfect matching}\label{sss1}

In this section we study Sterboul--Deming graphs that admit a perfect matching. We present characterizations both for the general case of graphs with a perfect matching and for the particular case where such matching is unique. We also provide an algorithm to obtain the SD--KE decomposition in graphs with a unique perfect matching.

\begin{theorem}\label{Thm1}
	A graph with a unique perfect matching is a Sterboul--Deming
	graph if and only if it has no leaves.
\end{theorem}

\begin{proof}
	If $G$ is a Sterboul--Deming graph with a unique perfect matching,
	then trivially it has no leaves, since $\delta(G)\ge2$. Conversely,
	suppose that $G$ is a graph with a unique perfect matching $M$
	and without leaves. Let $v\in V(G)$ and let $P=v_{1},v_{2},\dots,v_{k}$
	be an $mm$-path starting at $v=v_{1}$ of maximum length; then
	$k$ is even. Since the graph has no leaves, there exists $1\le i\le k-1$
	such that $v_{i}\in N(v_{k})$. As there are no even alternating
	cycles, $i$ must be even, that is, an odd cycle is formed. Repeating
	the same reasoning for $M(v)$ we obtain an $mm$-path $Q$ starting
	at $M(x)$ and ending at the base of an $M$-blossom $B$. Finally,
	the paths $P$ and $Q$ together with the blossom $B$ and the edge
	$v_{k}v_{i}$ form an $M$-Jposy containing $v$, that is, $v\in SD(G)$. 
\end{proof}

If $G$ is a graph with a unique perfect matching $M$ having a leaf
$v$, then $v,M(v)\in KE(G)$. Moreover, if $G^{\P}:=\left(G-v\right)-M(v)$,
then $SD(G^{\P})=SD(G)$, since $G^{\P}$ also has a unique perfect
matching. Therefore, we can continue this process and, from \cref{Thm1},
we obtain the following algorithm to compute an SD--KE decomposition
in a graph with a unique perfect matching. 

\begin{algorithm}[H]
	\DontPrintSemicolon
	\KwIn{Graph $G=(V,E)$ and the unique perfect matching $M$ of $G$}
	\KwOut{Sets $KE$ and $SD$, the SD--KE partition of $G$}
	\;
	
	Initialization: $KE \gets \emptyset$ \tcp*{Initial KE set}
	\;
	
	\While{$G$ has a vertex $v$ of degree $1$ in $G$}{
		$KE \gets KE \cup \{v, M(v)\}$ \tcp*{Add vertex and its partner}\;
		$G \gets \left(G-v\right)-M(v)$ \tcp*{Delete the vertices $v$ and $M(v)$ from $G$}
	}
	$SD \gets V-KE$ \tcp*{Define SD set}
	\;
	
	\Return $(SD, KE)$ \tcp*{Return the SD--KE partition of $G$}
	\caption{SD--KE decomposition algorithm for graphs with a unique perfect matching}
\end{algorithm}

\begin{figure}[h!]
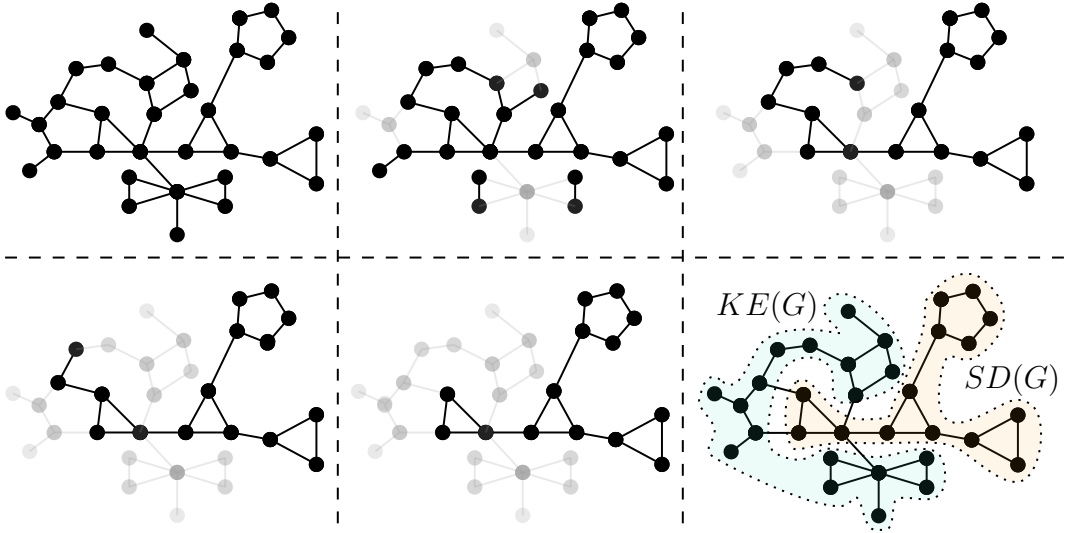

	
	\begin{center}

		\tikzset{every picture/.style={line width=0.75pt}} 
		


	\end{center}
	\caption{Illustration of Algorithm~1
	}
	\label{Fig1ura1}
	
\end{figure}

The previous algorithm allows us to obtain the SD--KE decomposition in graphs with a unique perfect matching, without the need to explicitly find such matching, since at each iteration only the leaf and its neighbor are removed (see \cref{Fig1ura1}). \cref{Thm1} can also be stated in the following equivalent form, which is useful for verifying whether a graph is $SD$. The advantage of this formulation lies in checking the uniqueness condition, which can be verified by inspecting any perfect matching of the graph.

\begin{theorem}
	Let $G$ be a graph with a perfect matching. Then the following statements are equivalent:
	\begin{enumerate}
		\item $G$ is a Sterboul--Deming graph with a unique perfect matching.
		\item $G$ has no $M$-alternating cycle for some perfect matching $M$ and has no leaves.
		\item $G$ has no $M$-alternating cycle for every perfect matching $M$ and has no leaves.
	\end{enumerate}
\end{theorem}

\begin{proof}
	$1\Rightarrow2)$ By \cref{Thm1}, $G$ has no leaves and the unique maximum matching does not leave alternating cycles, since it is unique. 
	
	$2\Rightarrow3)$ Let $M$ and $M^{\P}$ be two perfect matchings. Suppose that $G$ has no $M$-alternating cycles; then $M\triangle M^{\P}=\emptyset$, and consequently $M=M^{\P}$. 
	
	$3\Rightarrow1)$ A similar argument shows that $G$ has a unique perfect matching in this case. Moreover, since $G$ has no leaves, by \cref{Thm1} it follows that $G$ is a Sterboul--Deming graph.
\end{proof}

In the remainder of this section, we provide characterizations of Sterboul--Deming graphs with a perfect matching (not necessarily unique). The key step here is to rely on a result from \cite{agusvalenota}, which establishes the equivalence between the Larson independence decomposition \cite{larson2011critical} and the SD--KE decomposition in graphs with a perfect matching.

An independent set $I$ is a \emph{critical independent set} if $\left|I\right|-\left|N(I)\right|\ge\left|J\right|-\left|N(J)\right|$ for any independent set $J$. A critical independent set is \emph{maximum} if it has maximum cardinality. In \cite{larson2011critical}, Larson introduced the following decomposition theorem.

\begin{theorem}\label{larsonthm}
	For any graph $G$, there exists a unique set $L(G)\subseteq V(G)$ such that
	\begin{enumerate}
		\item $\alpha(G)=\alpha(G[L(G)])+\alpha(G[V(G)-L(G)])$,
		\item $G[L(G)]$ is a König--Egerváry graph,
		\item for every non-empty independent set $I$ in $G[V(G)-L(G)]$, we have $\left|N(I)\right|>\left|I\right|$, and
		\item for every maximum critical independent set $J$ of $G$, we have $L(G)=J\cup N(J)$.
	\end{enumerate}
\end{theorem}

In \cite{agusvalenota}, it was shown that for every graph with a perfect matching, the Larson decomposition and the SD--KE decomposition coincide.

\begin{theorem}[\label{agusvale}\cite{agusvalenota}]
	If $G$ is a graph with a perfect matching, then $KE(G)=L(G)$.
\end{theorem}

The following theorem provides a Hall--Tutte type characterization for Sterboul--Deming graphs with a perfect matching. Specifically, a graph is a Sterboul--Deming graph if and only if it satisfies Hall’s condition with excess on the independent sets and Tutte’s condition on every vertex subset.

\begin{theorem}\label{agustinachat}
	A graph $G$ is a Sterboul--Deming graph with a perfect matching if and only if
	\begin{enumerate}
		\item $\left|N(S)\right|>\left|S\right|$ for every non-empty $S\in\Omega^{*}(G)$, and
		\item $\textnormal{odd}(G-S)\le\left|S\right|$ for every $S\subseteq V(G)$.
	\end{enumerate}
\end{theorem}

\begin{proof}
	Let $G$ be a Sterboul--Deming graph with a perfect matching.
	Then, by \cref{agusvale}, $SD(G)=V(G)-L(G)=V(G)$. The first condition follows from item~3 in \cref{larsonthm}, while the second follows from Tutte’s theorem.
	
	Conversely, suppose that (1) and (2) hold for a graph $G$.
	By Tutte’s theorem, $G$ has a perfect matching, and by (1),
	the empty set is the only critical independent set. Hence, by
	\cref{larsonthm} and \cref{agusvale}, we have
	$$KE(G)=L(G)=\emptyset\cup N(\emptyset)=\emptyset,$$
	that is, $G$ is a Sterboul--Deming graph.
\end{proof}

A spanning subgraph $S$ of $G$ is called a \emph{Sachs subgraph} of $G$ if every component of $S$ is either $K_2$ or a cycle. Let $\Sachs{G}$ denote the set of all Sachs subgraphs of $G$. If $G$ has a perfect matching, then for every $M \in \mathcal{M}(G)$ we have $M \in \Sachs{G}$. In particular, if $G$ has a perfect matching, then $|\Sachs{G}| \geq 1$ \cite{s3}. 

In \cite{ksachscritical}, the concept of a \emph{$k$-Sachs critical graph} was introduced: a graph $G$ is $k$-Sachs critical if $G-S$ contains at least one Sachs subgraph for every subset $S\subseteq V(G)$ with $\left|S\right|=k$. This concept yields the following characterization of Sterboul--Deming graphs (\cref{panejau}).

\begin{theorem}[\cite{ksachscritical}\label{ksacriticahar}]
	A graph is $k$-Sachs critical if and only if
	$$i(G-S)\le \left|S\right|-k$$
	for every $S\subseteq V(G)$ such that $\left|S\right|\ge k$.
\end{theorem}

\begin{theorem}\label{panejau}
	Let $G$ be a graph with a perfect matching. Then $G$ is
	a Sterboul--Deming graph if and only if $G$ is a $1$-Sachs critical
	graph.
\end{theorem}

\begin{proof}
	Suppose that $G$ is a Sterboul--Deming graph with a perfect matching.
	Let $\emptyset\neq S\subseteq V(G)$ and let $I$ be the set of isolated vertices
	in $G-S$; then $N(I)\subseteq S$. Hence, by \cref{agustinachat},
	\[
	i(G-S)=\left|I\right|<\left|N(I)\right|\le\left|S\right|.
	\]
	Therefore, by \cref{ksacriticahar}, $G$ is a $1$-Sachs critical graph.	
	Conversely, suppose that $G$ is a $1$-Sachs critical graph
	with a perfect matching. Since $G$ has a perfect matching,
	it satisfies item~2 of \cref{agustinachat}. On the other hand,
	let $\emptyset\neq S\subseteq\Omega^{*}(G)$. Then, by \cref{ksacriticahar},
	\[
	\left|S\right|\le i(G-N(S))\le\left|N(S)\right|-1.
	\]	
	Thus, item~1 of \cref{agustinachat} is satisfied,
	as desired.
\end{proof}

\begin{theorem}
	[\cite{edmonds1965paths,gallai1964maximale}\label{ge} Gallai-Edmonds Structure Theorem]
	Given a graph $G$, let 
	\begin{align*}
		D(G) & :=\{ v : \textnormal{there exists a maximum matching missing } v \}, \\
		A(G) & :=\{ v : v \textnormal{ is a neighbor of some } u \in D(G), \textnormal{ but } v \notin D(G) \}, \\
		C(G) & := V(G) \setminus (D(G) \cup A(G)).
	\end{align*}
	\noindent If $G_1,\dots,G_k$ are the components of $G[D(G)]$
	and $M$ is a maximum matching of $G$, then
	\begin{enumerate}
		\item $M$ covers $C(G)$ and matches $A(G)$ into distinct components of $G[D(G)]$.
		\item Each $G_i$ is a factor-critical graph, and $M$ restricted to $G_i$
		is a near-perfect matching. 
		\item If $\emptyset\neq S\s A$, then $N_{G}(S)$ has a vertex in at least
		$\left|S\right|+1$ of $G_{1},\dots,G_{k}$. 
	\end{enumerate}
\end{theorem}

For a graph $G$, we define $\textnormal{Posy}(G)$ as the subset
of vertices of $V(G)$ that belong to an $M$-Tposy for some maximum
matching $M$. Similarly, we define $\textnormal{Flower}(G)$ as
the set of vertices contained in some Tflower of $G$. Trivially, a
Sterboul--Deming graph with a perfect matching satisfies
$\textnormal{Posy}(G)-\textnormal{Flower}(G)=V(G)$,
since in this case we have $\textnormal{Flower}(G)=\emptyset$.
The following theorem shows that the converse of this statement
is also true. For this, we make use of the following auxiliary lemma.

\begin{lemma}[\label{kevinsdkege}\cite{kevinsekege}]
	Let $G$ be a graph and let $v\in(D(G)\cup A(G))\cap SD(G)$.
	Then there exists a maximum matching $M$ such that $v$ belongs to an $M$-Tflower.
\end{lemma}

If a Sterboul--Deming graph has no perfect matching, then there exists
a vertex $v\in(D(G)\cup A(G))\cap SD(G)$. Hence, by \cref{kevinsdkege},
$\textnormal{Posy}(G)-\textnormal{Flower}(G)\neq V(G)$,
since $\textnormal{Flower}(G)\neq\emptyset$.
That is, we have the following result.

\begin{theorem}\label{1ds1a}
	A graph $G$ is a Sterboul--Deming graph with a perfect
	matching if and only if $\textnormal{Posy}(G)-\textnormal{Flower}(G)=V(G)$.
\end{theorem}

Similarly, from \cref{kevinsdkege} we deduce the following theorem,
which is analogous to \cref{1ds1a} for graphs without a perfect matching and $C(G)=\emptyset$.

\begin{theorem}
	A graph $G$ is a Sterboul--Deming graph with $C(G)=\emptyset$
	if and only if $V(G)=\textnormal{Flower}(G)$.
\end{theorem}

\section{Sterboul-Deming graphs}\label{sss2}

The first goal of this section is to prove a characterization theorem
for Sterboul--Deming graphs without a perfect matching. This result
is a reduction theorem: a graph $G$ is a Sterboul--Deming graph if and only if
its reduced form \emph{$\mathcal{R}(G)$} is a Sterboul--Deming graph.
This reduction arises naturally from the Gallai--Edmonds decomposition
of the graph. We begin with some preliminary ideas before stating
the main theorem.

\begin{theorem}[\cite{jaume2025confpart3}]
	$G[KE(G)]$ is a König--Egerváry graph, for every graph $G$.
\end{theorem}

\begin{theorem}[\label{420thmkl}\cite{jaume2025confpart3}]
	If $G$ is a graph with a perfect matching $M$, then for every vertex $v\in SD(G)$
	there exists an $M$-Jposy of $G$ that contains $v$.
\end{theorem} 

From \cref{420thmkl} the following observation follows immediately.

\begin{lemma}\label{asdmasd891}
	If $G$ is a Sterboul--Deming graph with a perfect matching
	$M$, then for every $v\in V(G)$ the following statements hold:
	\begin{itemize}
		\item There exists an $mm$-walk from $v$ to the base of an $M$-blossom
		in $G$.
		\item There exists an $nm$-walk from $v$ to the base of an $M$-blossom
		in $G$.
	\end{itemize}
\end{lemma}

Let $G$ be a graph and let $D(G),A(G),C(G)$ be its Gallai--Edmonds
decomposition. Consider the nontrivial components
$D_{1},\dots,D_{k}$ of $G[D(G)]$. We now define the following reduction of $G$:
first, remove from $G$ all the vertices belonging to the nontrivial
components of $G[D(G)]$, that is, remove the vertices in
$D_{1},\dots,D_{k}$. Then add $k$ disjoint triangles
$T_{1},\dots,T_{k}$ (copies of $K_3$), one for each $D_i$.
For each triangle $T_i$, select one of its three new vertices and
denote it by $t_i$, for $i=1,\dots,k$.
Next, add edges of the form $t_i x_i$, where
$x_i\in N_{G}(D_i)\cap A(G)$, for every $i=1,\dots,k$.
The resulting graph is called the reduced form of $G$ and is denoted
by $\mathcal{R}(G)$. 
In short, $\mathcal{R}(G)$ is obtained from $G$ by collapsing each component
$D_{1},\dots,D_{k}$ into a triangle and transferring the original
connections to a single vertex in each triangle (see \cref{formareducida}).

\begin{figure}[h!]
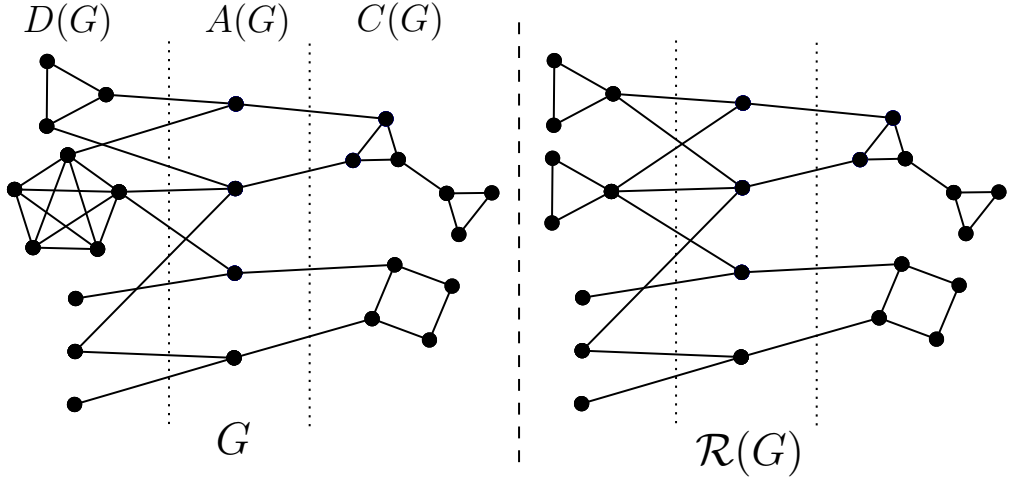

	
	\begin{center}

		\tikzset{every picture/.style={line width=0.75pt}} 
		


	\end{center}
	\caption{Illustration of the reduced form $\mathcal{R}(G)$ of $G$
	}
	\label{formareducida}
	
\end{figure}

By the Gallai--Edmonds decomposition theorem, we have the following observation.

\begin{observation}
	For every graph $G$, $A(G)=A(\mathcal{R}(G))$ and
	$C(G)=C(\mathcal{R}(G))$. In other words, $G$ and $\mathcal{R}(G)$
	have almost the same Gallai--Edmonds decomposition.
\end{observation}

\begin{lemma}[\label{kev912de}\cite{kevinsekege}]
	Every vertex belonging to a nontrivial component
	of $G[D(G)]$ lies in $SD(G)$.
\end{lemma} 

By considering the symmetric difference of two maximum matchings, the following well-known fact is easily obtained.

\begin{lemma}\label{matchisns12agus}
	For every maximum matching $M$ of a graph $G$
	and every vertex $v\in D(G)$, there exists an $mn$-path from $v$
	to an unsaturated vertex of $M$, possibly trivial.
\end{lemma}

A maximum matching $M$ of $G$ naturally induces a maximum matching
$\mathcal{R}(M)$ in $\mathcal{R}(G)$ as follows:
consider the set of edges
$M^{\P}:=M\cap E(\mathcal{R}(M))$, that is, the edges of $M$ having no
endpoints in any nontrivial component of $G[D(G)]$.
For each removed edge of the form $ad\in M$, with
$a\in A(G)$ and $d\in D(G)$ belonging to a nontrivial component,
we add to $M^{\P}$ the edge $at$, where $t$ is the distinguished vertex
in $\mathcal{R}(G)$ representing the component of $G[D(G)]$ containing $d$.

Taking into account \cref{kev912de}, the following result shows that
the graph $\mathcal{R}(G)$ preserves all the information of the
SD--KE decomposition of $G$.

\begin{theorem}\label{ilustr12a}
	For every graph $G$, we have $KE(G)=KE(\mathcal{R}(G))$.
\end{theorem} 

\begin{proof}
	To prove that $KE(G)=KE(\mathcal{R}(G))$, by \cref{kev912de},
	it is enough to show that every vertex in $SD(G)$ not lying in a nontrivial
	component of $G[D(G)]$ belongs to $SD(\mathcal{R}(G))$,
	and conversely, that every vertex in $SD(\mathcal{R}(G))$ not lying
	in a nontrivial component of $\mathcal{R}(G)[D(\mathcal{R}(G))]$
	belongs to $SD(G)$. 
	
	Let $x\in SD(G)$ be a vertex not contained in any nontrivial component
	of $G[D(G)]$. 
	
	\textbf{Case 1.} Suppose that $x$ lies in an $M$-Eflower for
	some maximum matching $M$ of $G$. Let the blossom of the flower
	be $B=b_{1},\dots,b_{n}$ with $b_{1}=b_{n}$, and let the stem be
	$P=p_{1},\dots,p_{k}$ with $p_{k}=b_{1}$. It is easy to see that every vertex of
	the blossom belongs to $D(G)$, hence $x$ lies in the stem of the flower. 
	
	\textbf{Case 1.1.} First, suppose that $x=p_{1}$.
	Then $p_{1}$ is an isolated vertex of $G[D(G)]$.
	Consequently, $p_{2}\in A(G)$ and $p_{3}\in D(G)$.
	If $p_{3}$ lies in a nontrivial component of $G[D(G)]$,
	it is easy to see that there exists an $nm$--$\mathcal{R}(M)$ path
	in $\mathcal{R}(G)$ between $x$ and the base of an
	$\mathcal{R}(M)$-blossom. If $p_{3}$ is an isolated vertex of
	$G[D(G)]$, then $p_{4}\in A(G)$ and $p_{5}\in D(G)$.
	We repeat the same argument for $p_{5}$ as we did for $p_{3}$.
	In this way, at some step we obtain a vertex $p_i$ with $i$
	odd such that $p_i$ belongs to a nontrivial component of $G[D(G)]$,
	since the base of the blossom satisfies this property.
	Therefore, there exists an $nm$--$\mathcal{R}(M)$ path
	in $\mathcal{R}(G)$ between $x$ and the base of an
	$\mathcal{R}(M)$-blossom.
	Since $x$ is unsaturated by $\mathcal{R}(M)$, we have that
	$x$ lies in an $\mathcal{R}(M)$-Eflower of $\mathcal{R}(G)$,
	and hence $x\in SD(\mathcal{R}(G))$. 
	
	Now suppose that $x\neq p_{1}$.
	Without loss of generality, assume that $x$ is at an odd distance
	from the base of the blossom; then $M(x)$ is at an even distance
	from the base of the blossom (see \cref{Figurasa1x}).
	We will prove that both $x$ and $M(x)$ belong to $SD(\mathcal{R}(G))$.

	\begin{figure}[h!]
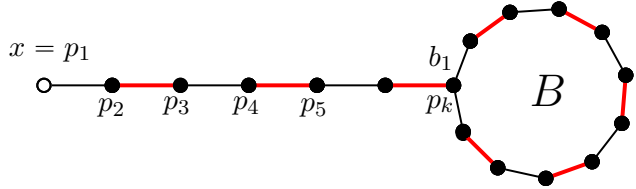

		
		\begin{center}

			\tikzset{every picture/.style={line width=0.75pt}} 
			


		\end{center}
		\caption{Illustration of the proof of \cref{ilustr12a}
		}
		\label{Figurasa1x}
		
	\end{figure}
	
\textbf{Case 1.2.} Then $M(x)\in D(G)$ and $x\in A(G)$. 
Analogously to the reasoning in Case~1.1, we can prove that there exists an 
$mm$--$\mathcal{R}(M)$ path in $\mathcal{R}(G)$ between $x$ and the base 
of an $\mathcal{R}(M)$-blossom. 

We now perform a similar reasoning moving to the left of $x$. 
Suppose that $x=p_i$ for some $i$. Consider the following cases:

\textbf{Case 1.3.} If $p_{i-1}\in A(G)$, then $p_{i-2}\in D(G)$.
If moreover $p_{i-2}$ belongs to a nontrivial component of $G[D(G)]$,
then there exists an $nm$--$\mathcal{R}(M)$ path in $\mathcal{R}(G)$
between $x$ and the base of an $\mathcal{R}(M)$-blossom. Hence $x$
lies in an $\mathcal{R}(M)$-Jposy of $\mathcal{R}(G)$.
Otherwise, suppose that $p_{i-2}$ is an isolated vertex of $G[D(G)]$; 
then $p_{i-3}\in A(G)$ and $p_{i-4}\in D(G)$.
Repeating the argument with $p_{i-4}$ instead of $p_{i-2}$, we eventually
obtain that $x$ belongs to an $\mathcal{R}(M)$-Jposy of $\mathcal{R}(G)$,
or possibly to an $\mathcal{R}(M)$-Jflower of $\mathcal{R}(G)$ when
this process terminates at $p_1\in D(G)$, using \cref{asdmasd891}. 

\textbf{Case 1.4.} If $p_{i-1}\in D(G)$ and moreover $p_{i-1}$ lies
in a nontrivial component of $G[D(G)]$, then by \cref{matchisns12agus}
we obtain that $x$ lies in an $\mathcal{R}(M)$-Jflower of $\mathcal{R}(G)$.
If $p_{i-1}$ is an isolated vertex of $G[D(G)]$, then $p_{i-2}\in A(G)$.
If also $p_{i-3}\in A(G)$ we repeat Case~1.3; 
if $p_{i-3}\in D(G)$ we repeat the current case.

\medskip

\textbf{Case 2.} Suppose that $x$ lies in an $M$-Tposy for some
maximum matching $M$ of $G$. Then there exists (up to a possible
rotation of $M$) an $mm$--$M$ path $P=p_{1},\dots,p_{k}$ in
$G$ connecting $x=p_{1}$ to the base of an $M$-blossom whose
intersection is only the base, say $B=p_{k},p_{k+1},\dots,(p_{n}=p_{k})$.
We will show that 
\begin{enumerate}
	\item there exists an $mm$--$\mathcal{R}(M)$ path in $\mathcal{R}(G)$
	between $x$ and the base of an $\mathcal{R}(M)$-blossom, or 
	\item there exists an $mn$-walk from $x$ to an unsaturated vertex 
	of $\mathcal{R}(M)$ in $\mathcal{R}(G)$. 
\end{enumerate}
The same holds for $M(x)$, but the second condition cannot occur
simultaneously for $x$ and $M(x)$. Therefore, both $x$ and $M(x)$
belong to an $\mathcal{R}(M)$-Jposy or an $\mathcal{R}(M)$-Eflower of 
$\mathcal{R}(G)$, and thus $x,M(x)\in SD(\mathcal{R}(G))$. 

Consider the following three subcases.

\textbf{Case 2.1.} Suppose that $x\in C(G)$. Then $p_{2}\in C(G)$.
If $p_{3}\in C(G)$, we repeat this case; otherwise, suppose that
$p_{3}\in A(G)$, then $p_{4}\in D(G)$. If $p_{4}$ lies in a nontrivial
component of $G[D(G)]$, we are done. Otherwise, suppose that $p_{4}$
is an isolated vertex of $G[D(G)]$. Then $p_{5}\in A(G)$ and $p_{6}\in D(G)$.
Repeating this argument, we either finish the case or obtain that no vertex 
of $V(B)\cup V(P)$ lies in a nontrivial component of $G[D(G)]$, 
which also completes the proof of this case. Here $x$ (and analogously $M(x)$)
satisfies condition (1).

\medskip

\textbf{Case 2.2.} Suppose that $x\in A(G)$; then $p_{2}\in D(G)$,
and we repeat the argument of Case~2.1. Here $x$ satisfies (1), 
and $M(x)$ is analyzed as in Case~2.3. 

\medskip

\textbf{Case 2.3.} Suppose that $x\in D(G)$; then $p_{2}\in A(G)$.
By Case~2.2, $M(x)$ satisfies (1). If $p_{3}\in C(G)$, we repeat
Case~2.1 and prove that $x$ satisfies (1). If $p_{3}\in A(G)$, we
repeat this case. Suppose that $p_{3}\in D(G)$. If $p_{3}$ lies in a
nontrivial component of $G[D(G)]$, then by \cref{matchisns12agus} 
$x$ satisfies (2), and we are done. Otherwise, suppose that $p_{3}$
lies in a trivial component of $G[D(G)]$; then $p_{4}\in A(G)$.
Repeating this argument, we either finish the case or find that no vertex
of $V(B)\cup V(P)$ lies in a nontrivial component of $G[D(G)]$, 
which also completes the case. 

The proof that a vertex $x\in SD(\mathcal{R}(G))$ not belonging to
a nontrivial component of $\mathcal{R}(G)[D(\mathcal{R}(G))]$
also lies in $SD(G)$ is essentially the same as above.
\end{proof}

As a corollary, we obtain the following characterization of
Sterboul--Deming graphs.

\begin{theorem}
A graph $G$ is a Sterboul--Deming graph if and
only if $\mathcal{R}(G)$ is a Sterboul--Deming graph.
\end{theorem}

For a positive integer $k$, a regular spanning subgraph in which
every vertex has degree $k$ is called a \emph{$k$-factor}.
Following the notation of \cite{akiyama2011factors}, we say that a
graph $G$ is a $\{C_{n}:n=3,5,\dots\}$-factor if $G$ has a $2$-factor
$F$ in which each cycle of $F$ is odd. In \cite{abreu2014odd},
odd $2$-factored graphs were studied; these are graphs in which
every $2$-factor consists solely of odd cycles, and a constructive
method for generating such graphs was proposed.
Moreover, partial characterizations were obtained,
providing conditions to determine when a cyclically
$4$-edge-connected graph is an odd $2$-factored graph.
Here we show that every graph with a 
$\{C_{n}:n=3,5,\dots\}$-factor, and in particular every odd
$2$-factored graph, is a Sterboul--Deming graph
(see \cref{desc-cicos-impares}). 
This result follows from the more general statement given in \cref{thm410ll}.

\begin{figure}[h!]
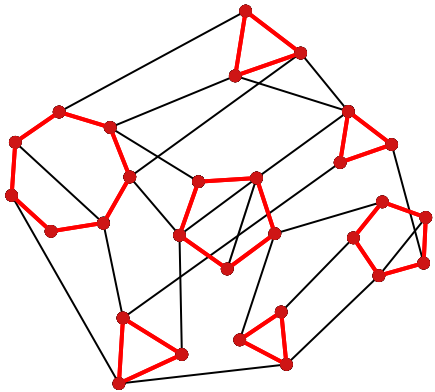

	
	\begin{center}

		\tikzset{every picture/.style={line width=0.75pt}} 
		


	\end{center}
	\caption{Example of a graph $G$ with a $\{C_{n}:n\textnormal{ odd}\}$-factor.
		Note that $G$ is a Sterboul--Deming graph.
		This is an example of \cref{thm410ll}.
	}
	\label{desc-cicos-impares}
	
\end{figure}
\begin{lemma}
	Let $G$ be a König--Egerváry graph, and let $S$ be a
	Sachs subgraph of $G$. Then $S$ contains no odd cycles.
\end{lemma}

\begin{proof}
	Suppose that $G$ is a König--Egerváry graph, and
	$S$ is a Sachs subgraph of $G$ containing an odd cycle. 
	If $H$ is an even component of $S$, then we need half of the vertices 
	of $H$ to cover its edges $E(H)$. If $H$ is an odd component, 
	more than half of its vertices are needed. 
	Therefore, since $S$ contains an odd cycle, we obtain 
	$\tau(G)>\frac{|G|}{2}$, and hence 
	$\tau(G)>\frac{|G|}{2}\ge\mu(G)$, which is a contradiction.
\end{proof}

\begin{lemma}[\label{123qwe2}\cite{kevindetgeneralsdke}]
	If $S\in\textnormal{Sachs}(G)$, then $E(S)\cap\partial(SD(G))=\emptyset$. 
\end{lemma}

\begin{theorem}\label{thm410ll}
	Let $S$ be a Sachs subgraph of $G$, and let $C$ be an
	odd cycle of $G$. Then $V(C)\subseteq SD(G)$. 
\end{theorem}

\begin{proof}
	By \cref{123qwe2}, $E(S)\cap\partial(SD(G))=\emptyset$.
	Thus, for every component $H$ of $S$, we have either 
	$V(H)\subseteq KE(G)$ or $V(H)\subseteq SD(G)$.
	Therefore, $S^{\P}:=S[KE(G)]$ is a Sachs subgraph
	of $G^{\P}:=G[KE(G)]$. 
	Since $G^{\P}$ is a König--Egerváry graph, $S^{\P}$
	has no odd cycles, and consequently $V(C)\subseteq SD(G)$.
\end{proof}

\begin{corollary}\label{corolarioasdarriba2}
	If every vertex of a graph $G$ lies in an odd cycle
	of a Sachs subgraph of $G$, then $G$ is a Sterboul--Deming graph.
\end{corollary}

In the graph shown in \cref{FiguraWIKI}, it is easy to see that every vertex
lies in an odd cycle of a Sachs subgraph of $G$; hence $G$ is a
Sterboul--Deming graph.

\begin{figure}[h!]
	
	\begin{center}

		\tikzset{every picture/.style={line width=0.75pt}} 
		


	\end{center}
	\caption{Illustration of \cref{corolarioasdarriba2}}
	\label{FiguraWIKI}
	
\end{figure}

On the other hand, since a $\{C_{n}:n=3,5,\dots\}$-factor of a graph
$G$ is a Sachs subgraph of $G$, we have the following.

\begin{corollary}
	Every graph with a $\{C_{n}:n=3,5,\dots\}$-factor
	is a Sterboul--Deming graph.
\end{corollary}

As an immediate corollary, every Hamiltonian graph of odd order is
a Sterboul--Deming graph. A complete graph of order $4$ is a Sterboul--Deming
graph, and therefore, by \cref{thm410ll}, every complete graph of order
at least $3$ is a Sterboul--Deming graph. The Petersen graph is also
a Sterboul--Deming graph. In 1970, Hajnal and Szemerédi \cite{hajnal1970proof}
proved the following.

\begin{theorem}[\cite{hajnal1970proof}\label{erdosconj}]
	Every graph $G$ with $\left|G\right|=n=ks$ and 
	$\delta(G)\ge\left(\frac{s-1}{s}\right)n$
	has a $\{K_{s}\}$-factor. 
\end{theorem}

From the result of Hajnal and Szemerédi and from \cref{thm410ll}, we obtain the following.

\begin{corollary}
	Let $s>2$. Then every graph $G$ with $\left|G\right|=n=ks$ and 
	$\delta(G)\ge\left(\frac{s-1}{s}\right)n$ 
	is a Sterboul--Deming graph. 
\end{corollary}

\begin{proof}
	If $s$ is odd, the result follows from \cref{thm410ll} and \cref{erdosconj}. 
	If $s$ is even, the result follows directly from \cref{safe} and \cref{erdosconj}.
\end{proof}

\section*{Open problems}

\begin{problem}
	Find a characterization of the type of \cref{agustinachat} for Sterboul--Deming
	graphs without a perfect matching.
\end{problem}

\begin{problem}
	Can the SD--KE decomposition be determined solely from the $\{C_{n}\}$-factors
	of the graph?
\end{problem}

\section*{Acknowledgments}

This work was partially supported by Universidad Nacional de San Luis (Argentina), PROICO 03-0723, MATH AmSud, grant 22-MATH-02, Agencia I+D+i (Argentina), grants PICT-2020-Serie A-00549 and PICT-2021-CAT-II-00105, CONICET (Argentina) grant PIP 11220220100068CO.

\section*{Declaration of generative AI and AI-assisted technologies in the writing process}
During the preparation of this work the authors used ChatGPT-3.5 in order to improve the grammar of several paragraphs of the text. After using this service, the authors reviewed and edited the content as needed and take full responsibility for the content of the publication.

\section*{Data availability}

Data sharing not applicable to this article as no datasets were generated or analyzed during the current study.

\section*{Declarations}

\noindent\textbf{Conflict of interest} \ The authors declare that they have no conflict of interest.

\bibliographystyle{apalike}

\bibliography{TAGcitasV2025}

\end{document}